\documentclass[12pt]{amsart}
\numberwithin{equation}{section}
\usepackage{amsmath,amsthm,amsfonts,amscd,eucal}

\usepackage{xypic}
\usepackage{graphicx}

\usepackage{amssymb}
\def\cb{{\mathcal B}}

\def\ce{{\mathcal E}}
\def\cf{{\mathcal F}}

\def\ch{{\mathcal H}}

\def\ck{{\mathcal K}}

\def\cs{{\mathcal S}}

\def\ga{{\mathfrak A}} 
\def\gpb{{\mathfrak b}}

\def\ge{{\mathfrak E}}
\def\gf{{\mathfrak F}}
\def\gg{{\mathfrak G}}

\def\gq{{\mathfrak Q}}
\def\gar{{\mathfrak R}}

\def\gz{{\mathfrak Z}}

\def\bc{{\mathbb C}}

\def\bn{{\mathbb N}}
\def\bp{{\mathbb P}}
\def\br{{\mathbb R}}

\def\bz{{\mathbb Z}}

\def\a{\alpha}
\def\b{\beta}
\def\g{\gamma}  \def\G{\Gamma}

\def\eps{\varepsilon}

\def\l{\lambda}

\def\s{\sigma} 
\def\t{\tau}
\def\f{\varphi}  
  
\def\om{\omega} \def\Om{\Omega}

\def\id{\hbox{id}}
\def\ker{\hbox{Ker}}

\newtheorem{thm}{Theorem}[section]
\newtheorem{lem}[thm]{Lemma}

\newtheorem{cor}[thm]{Corollary}
\newtheorem{prop}[thm]{Proposition}

\theoremstyle{definition}
\newtheorem{rem}[thm]{Remark}
\newtheorem{defin}[thm]{Definition}

\def\spn{\mathop{\rm span}}

\def\idd{{1}\!\!{\rm I}}

\newcommand{\nn}{\nonumber}

\begin{document}

\title[monotone processes and ergodic theorems]
{Ergodic theorems in quantum probability: an application to the monotone stochastic processes}
\author{Vitonofrio Crismale}
\address{Vitonofrio Crismale\\
Dipartimento di Matematica\\
Universit\`{a} degli studi di Bari\\
Via E. Orabona, 4, 70125 Bari, Italy}
\email{\texttt{vitonofrio.crismale@uniba.it}}
\author{Francesco Fidaleo}
\address{Francesco Fidaleo\\
Dipartimento di Matematica\\
Universit\`{a} degli studi di Roma Tor Vergata\\
Via della Ricerca Scientifica 1, Roma 00133, Italy} \email{{\tt
fidaleo@mat.uniroma2.it}}
\author{Yun Gang Lu}
\address{Yun Gang Lu\\
Dipartimento di Matematica\\
Universit\`{a} degli studi di Bari\\
Via E. Orabona, 4, 70125 Bari, Italy}
\email{\texttt{yungang.lu@uniba.it}}
\date{\today}

\begin{abstract}

\vskip0.1cm\noindent We give sufficient conditions ensuring the strong ergodic property of unique mixing for $C^*$-dynamical systems arising from Yang-Baxter-Hecke quantisation. We discuss whether they can be applied to some important cases including Monotone, Boson, Fermion and Boolean $C^*$-algebras in a unified version. The Monotone and the Boolean cases are treated in full generality, the Bose/Fermi cases being already widely investigated. In fact, on one hand we show that the set of stationary stochastic processes are isomorphic to a segment in both the Monotone and Boolean situations, on the other hand the Boolean processes enjoy the very strong  property of unique mixing with respect to the fixed point subalgebra and the Monotone ones do not.
\end{abstract}

\maketitle

\section{introduction}
\label{sec1}
One of main tasks in studying dynamical systems consists in finding suitable conditions yielding some convergence to the equilibrium. From the classical viewpoint, given a dynamical system $(\Om,T)$, consisting of a compact Hausdorff space $\Om$ and a homeomorphism $T$, one finds, among the most prominent, the unique ergodicity, which deals with the existence of a unique invariant Borel probability measure $\mu$ for the dynamics $T$. Such a strong property possesses many generalisations, even if one considers a non-commutative, or quantum, dynamics, given by the couple $(\ga, \alpha)$ made of a $C^*$-algebra $\ga$ and an automorphism $\alpha$. In fact, the fast recent growing up of such investigations has shown that, besides the unique ergodicity, other notions of convergence to a unique invariant state can be studied, which are meaningless or reduced to the usual one in the classical case. Among them, we mention the unique weak mixing and the unique mixing, see \cite{F22, {FM1}, {FM2}}. In addition, the emergence of non trivial fixed point subalgebras of $\ga$ for the action of $\alpha$, has suggested the analogue of the cited above properties when there exists an invariant conditional expectation onto such subalgebra which plays the role of the state in the previous case.

Very recently, it has been pointed out that there is a one-to-one correspondence between quantum stochastic processes based on a $C^*$-algebra, and states on the free product $C^*$-algebra of the infinitely many copies of the same algebra \cite{CrFid, CrFid1}. This approach, based on the universal property of free product $C^*$-algebra, can be applied to several remarkable examples including the classical ones. Thus, quite naturally in many cases of interest, the investigation of stochastic processes can be achieved directly on concrete $C^*$-algebras seen as the quotient. This is the case of CCR and CAR algebras, the most known and investigated examples both for the natural applications in Quantum Physics (see \cite{BR1} and references cited therein), the Boolean algebra useful in quantum optics \cite{vW}, as well as the classical (i.e. commutative) case, covered by considering directly the free Abelian product which corresponds to the free product $C^*$-algebra factored out by the ideal generated by all the commutators \cite{CrFid1}.

It seems then natural to address the study of ergodicity for a family of dynamical systems whose set of algebras contains those listed above, and the automorphism is the simplest one, i.e. the shift. The natural action of the permutation group can be also considered achieving the so-called de Finetti-like Theorems, see e.g. \cite{ABCL, CrF, St2}.
Indeed, it is our goal to study such properties when $\ga$ is a concrete $C^*$-algebra whose generators realise the commutation relations between creators and annihilators in suitable Fock spaces, and $\alpha$ is the action of group $\mathbb{Z}$ (i.e. the shift), or of the group of permutations on the algebra. We mention that the study or ergodicity for algebras generated by elements having Fock representation as creators and annihilators has been already done for the case of $q$-commutation relations ($|q|<1$) in \cite{DykFid} and \cite{CrFid}. There it was shown that the shift and the permutations have both the strong ergodic property of unique mixing.

The goal of the present paper is to extend the investigation of ergodic states with respect to the action of the shift in the case of a $C^*$-algebra whose generators arise as creators and annihilators on the so-called $T$-deformed Fock space \cite{Boz}. This situation slightly differs from the $q$-deformed one. Namely, the selfadjoint Yang-Baxter operator $T$ on the tensor product of a Hilbert space with himself which originates the $q$-deformation is a strict contraction, whereas here we require it is of Hecke type, see below. This entails it is necessarily bounded, with norm bigger or equal to one \cite{Boz}. Such a quantisation scheme covers the CCRs, the CARs, as well as the Boolean or the Monotone commutation relations, among the most important examples. As it is well known that the free (or Boltzmann) commutation relations are a particular case of $T$ (i.e. $T=0$), one can say that, after our results, the description of the ergodic behaviour of some relevant dynamical systems giving rise to the five universal independencies in Quantum Probability (cf. \cite{BSc, {M2}}) has been made.
In particular, here  we find sufficient conditions under which achieve strong ergodic results for $T$-deformed $C^*$-algebras. Moreover, we check whether they can be applied in the Monotone case. The Boolean case is analysed in full generality too, and the Boson and Fermion cases are reviewed to provide a unified treatment. More in detail, the paper is organised as follows.

In Section \ref{prel} we introduce some definitions and known results about ergodicity and stochastic processes we will use throughout the successive parts.

Section \ref{YB} is devoted to Yang-Baxter quantisation. There we write down the natural condition that, if satisfied, ensures that a group acts as Bogoliubov automorphisms on the $T$-deformed $C^*$-algebra $\gar_T$.

In Section \ref{sec3}, we search for and give sufficient conditions under which $(\gar_T,\alpha)$ is uniquely mixing. These are summarised in Theorem \ref{mixing}.

Section \ref{monsec} deals with the Monotone stochastic processes. Here, after checking that some of the technical conditions presented in Section \ref{sec3} are satisfied, we recognise that, as a consequence of relations between Monotone creators and annihilators, Theorem \ref{mixing} cannot be applied. Furthermore, by a simple computation we find the Monotone $C^*$-algebra is not uniquely ergodic with respect to the shift. We show that this obstruction is not a fatal one. In fact, we first completely clarify the fine structure of the Monotone $C^*$-algebra. After, we use it in order to gain the exact structure of the convex set of shift invariant states (Theorem \ref{invmon}), which, quite surprisingly, results isomorphic to a segment of the real line.

Section \ref{exa} starts with the Anti-Monotone case which can be immediately brought back to the Monotone one. We furthermore show that it is possible to construct a $C^*$-dynamical system which can enjoy the unique mixing property with respect to the fixed point subalgebra, although in presence of another obstruction, that is the unboundedness of the annihilators, which cannot allow directly to apply Theorem \ref{mixing}. We exploit such a result for the $C^*$-algebra coming from the Bose commutation relations, by providing a pivotal example for potential physical applications. In such a section devoted to examples, we review the status of arts for the Boson and Fermi cases, concerning the ergodic properties.

Finally, in Section \ref{boolsec} the ergodic properties of Boolean stochastic processes are treated in full generality. Here, the general ergodic properties of Section \ref{sec3} cannot be applied too, because of the particular commutation relations involved. However, even in this case, after showing the fixed point subalgebras with respect to the shift and permutations are not trivial but equal, we are able to prove that the Boolean $C^*$-algebra and the shift are an example of unique mixing dynamical system. Finally, we find that the shift invariant states are isomorphic to a segment of the real line, exactly the same one it has been achieved in \cite{CrFid} for permutation invariant states.

\textbf{Acknowledgements.} The first and second named authors have been partially supported by Gruppo Nazionale per l' Analisi Matematica, la Probabilit\`{a} e le loro Applicazioni of Istituto Nazionale di Alta Matematica.

\section{preliminaries and notations}
\label{prel}
In this section we report for the convenience of the reader, some known definitions, notations and results. All the involved operators are considered bounded if it is not otherwise specified. In addition, the morphisms considered in the present paper preserve all the algebraic properties, included the $*$-operation, without any further mention.

\subsection{ergodic properties}

A $C^*$-dynamical system based on the group $G$, is a pair $(\ga,\a)$ made of $C^*$-algebra $\ga$ which always we suppose unital with unity $\idd$, and an action  
(i.e. a group homomorphism of $G$ into the group of all the automorphisms ${\rm Aut}(\ga)$)
$$
\a:g\in G\mapsto\a_g\in{\rm aut}(\ga)\,.
$$
The {\it fixed point subalgebra} is defined as
$$
\ga^G:=\{a\in\ga\mid \a_g(a)=a\,,\,\,g\in G\}\,.
$$
A state $\f\in\cs(\ga)$ is called $G$-invariant if $\f=\f\, \circ \a_g$ for each $g\in G$. The subset of the $G$-invariant states is denoted by $\cs_G(\ga)$. It is $*$-weakly compact, and its extremal points $\ce_G(\ga)$
are called {\it ergodic} states. For $(\ga, \a)$ as above and an invariant state $\f$ on $\ga$,
$(\pi_\f,\ch_\f,U_\f,\Omega_\f)$ is the Gelfand-Naimark-Segal (GNS for short) covariant
quadruple canonically associated to $\f$, see e.g. \cite{BR1, T1}.
By $\gz_\f:=\pi_\f(\ga)''\bigwedge\pi_\f(\ga)'$ we denote the center of $\pi_\f(\ga)''$.

\noindent We mention the action of the group of the permutations $\bp_J:=\bigcup\{\bp_I|I\subseteq J\, \text{finite}\}$ of
an arbitrary set leaving fixed all the elements in $J$, except a finite number of them. If $\bp_J$ is acting on $\ga$, an element
$\f\in\cs_{\bp_J}(\ga)$ is called {\it symmetric}.
We also mention the simplest case when the Abelian group $\bz$ made of the integer numbers is acting on a $C^*$-algebra. In fact, a (discrete) $C^*$-dynamical system is a pair
$(\ga,\a)$ based on a single automorphism $\a$ of $\ga$, which automatically generates the action of $\bz$. To achieve dissipative (i.e. non unitary) dynamics, one can suppose that $\a$ is merely a completely positive, identity preserving map. In the latter, only the monoid $\bn$ is naturally acting on $\ga$.

Suppose that $\cs_{\bz}(\ga)=\{\om\}$ is a singleton. Such a dynamical system is said to be {\it uniquely ergodic}. One can see that unique ergodicity is equivalent to
\begin{equation}
\label{eav}
\lim_{n\rightarrow+\infty} \frac{1}{n}\sum_{k=0}^{n-1}f(\alpha^k(a))=f(\idd)\omega(a)\,, \quad a\in \ga\,,f\in\ga^*\,,
\end{equation}
or again to
$$
\lim_{n\rightarrow+\infty} \frac{1}{n}\sum_{k=0}^{n-1} \alpha^k(a)=\omega(a)\idd, \,\,\,\,\,\, a\in \ga\,\,,
$$
pointwise in norm. Some natural generalisations of such a strong ergodic property are the following. The first ones concern to replace the ergodic average \eqref{eav} with
$$
\lim_{n\rightarrow+\infty} \frac{1}{n}\sum_{k=0}^{n-1}|f(\alpha^k(a))-f(\idd)\omega(a)|=0\,, \quad a\in \ga\,,f\in\ga^*\,,
$$
or simply
$$
\lim_{n\rightarrow+\infty} f(\alpha^n(a))=f(\idd)\omega(a)\,, \quad a\in \ga\,,f\in\ga^*\,,
$$
for some state $\om\in\cs(\ga)$ which is necessarily unique and invariant. In this case, $(\ga,\a)$ is called {\it uniquely weak mixing} or {\it uniquely mixing}, respectively. For all these cases,
$\ga^\bz=\bc\idd$, and the (unique) invariant conditional expectation onto the fixed point subalgebra is precisely $E(a)=\om(a)\idd$.

Another natural generalisation is to look at the fixed point subalgebra whenever it is nontrivial, and at the unique invariant conditional expectation onto such a subalgebra
$E^\bz:\ga\to\ga^\bz$, provided the last exists. The unique ergodicity, weak mixing, and mixing w.r.t. the fixed point subalgebra (denoted also as $E^\bz$-ergodicity, $E^\bz$-weak mixing, and $E^\bz$-mixing, $E^\bz$ being the invariant conditional expectation onto $\ga^\bz$ which necessarily exists and it is unique) are given by definition, for $a\in\ga$ and $f\in\ga^*$, by
\begin{align*}
&\lim_{n\rightarrow+\infty} \frac{1}{n}\sum_{k=0}^{n-1}f(\alpha^k(a))=f(E^\bz(a))\,,\\
&\lim_{n\rightarrow+\infty} \frac{1}{n}\sum_{k=0}^{n-1}|f(\alpha^k(a))-f(E^\bz(a))|=0\,,\\
&\lim_{n\rightarrow+\infty} f(\alpha^n(a))=f(E^\bz(a))\,.
\end{align*}
A systematic study of general strong ergodic properties of quantum dynamical systems including unique mixing and weakly mixing is contained in \cite{F22}, and some applications to free and $q$-deformed probability appeared in \cite{DykFid, FM1, FM2} . The reader is referred to those papers and the literature cited therein for further details.

\subsection{(quantum) stochastic processes}

We recall some notations and facts concerning the notion of quantum stochastic processes, firstly introduced in the seminal paper \cite{AFL}, which is suitable for Quantum Probability, but includes the classical case as a particular one (cf. \cite{CrFid, CrFid1}).

A stochastic process labelled by the index set $J$ is a quadruple
$\big(\ga,\ch,\{\iota_j\}_{j\in J},\Om\big)$, where $\ga$ is a $C^{*}$-algebra, $\ch$ is an Hilbert space,
the $\iota_j$'s are $*$-homomorphisms of $\ga$ in $\cb(\ch)$, and
$\Om\in\ch$ is a unit vector, cyclic for  the von Neumann algebra
$M:=\bigvee_{j\in J}\iota_j(\ga)$ naturally acting on $\ch$. The quadruple defining a given stochastic process is uniquely determined up to unitary equivalence.

The process is said to be {\it exchangeable} if, for each $g\in\bp_J$, $n\in\bn$, $j_1,\ldots j_n\in J$, $A_1,\ldots A_n\in\ga$
$$
\langle\iota_{j_1}(A_1)\cdots\iota_{j_n}(A_n)\Om,\Om\rangle
=\langle\iota_{g(j_{1})}(A_1)\cdots\iota_{g(j_{n})}(A_n)\Om,\Om\rangle.
$$
Suppose that $J=\bz$. The process is said to be {\it stationary} if  for each $n\in\bn$, $j_1,\ldots j_n\in\bz$, $A_1,\ldots A_n\in\ga$
$$
\langle\iota_{j_1}(A_1)\cdots\iota_{j_n}(A_n)\Om,\Om\rangle
=\langle\iota_{j_{1}+1}(A_1)\cdots\iota_{j_{n}+1}(A_n)\Om,\Om\rangle.
$$
In \cite{CrFid} it was established that there is
a one-to-one correspondence between quantum stochastic processes, either preserving or not the identity, and states on free product $C^*$-algebras, unital or not unital respectively. To simplify, we deal ever with the unital case, i.e. when $\ga$ has the unity $\idd$ and $\iota_j(\idd)=\idd$, $j\in J$.

For a stochastic process $\big(\ga,\ch,\{\iota_j\}_{j\in J},\Om\big)$ or for its corresponding state $\f\in\cs(\gf)$ on the free product $C^*$-algebra $\gf:={\bf *}_{j\in J}\ga$,
the {\it tail algebra} (or the {\it algebra at infinity} in physical language) is defined as
\begin{equation*}
\gz^\perp_\f:=\bigcap_{\begin{subarray}{l}I\subset J,\,
I \text{finite} \end{subarray}}\bigg(\bigcup_{\begin{subarray}{l}K\bigcap I=\emptyset,
\\\,\,\,K \text{finite} \end{subarray}}\bigg(\bigvee_{k\in K}\iota_k(\ga)\bigg)\bigg)''\,.
\end{equation*}
In this unified setting, exchangeable or stationary stochastic processes correspond to symmetric or shift invariant states on the free product $C^*$-algebra $\gf$, respectively.
If a process is exchangeable, it is automatically stationary, then it is meaningful to compare the tail, exchangeable and invariant algebras. It is a fundamental result of classical probability (cf. \cite{HS, O}) that
$$
\gz^\perp_\f=(\pi_\f(\gf)'')^{\bp_\bz}=(\pi_\f(\gf)'')^{\bz}\,.
$$
In quantum probability, such equalities do not hold in general. In fact, there are examples for which $\gz^\perp_\f\neq(\pi_\f(\gf)'')^{\bp_\bz}$ (cf. \cite{K}), and
$\gz^\perp_\f\neq(\pi_\f(\gf)'')^{\bz}$ (cf. \cite{Fbo}).

Processes  arising from a $C^*$-algebra $\gq=\gf/\sim$ generated by some closed two sides ideal (often generated by commutation relations in concrete cases), that is stochastic processes factoring through $\gq$, like $q$-deformed including the Bose/Fermi and the free (i.e. Boltzmann) cases, or Boolean and Monotone ones, can be viewed directly as states on $\gq$. From this, it is customary to call such processes directly as $q$-deformed
(including the Bose/Fermi and the free alternative), Boolean or Monotone, respectively.

\section{Yang-Baxter-Hecke quantisation}
\label{YB}
Let $\ch$ be a Hilbert space. A selfadjoint, not necessarily bounded operator $T: \ch \otimes \ch \rightarrow \ch \otimes \ch$ such that $T\geq -I$ and
$$
T_1T_2T_1=T_2T_1T_2\,\,,
$$
where $T_1:=T\otimes I$ and $T_2:=I\otimes T$ on $\ch\otimes \ch\otimes \ch$, is called a Yang-Baxter operator.
For each $n\in\mathbb{N}$, denote $\ch^{\otimes n}:=\underbrace{\ch\otimes \cdots \otimes \ch}_{n}$ and
$$
T_{k}:=\underbrace{I \otimes \cdots \otimes I}_{k-1\,\, \text{times}} \otimes T\otimes \underbrace{I \otimes \cdots \otimes I}_{n-k-1\,\, \text{times}}\,\,\, \text{on}\,\, \ch^{\otimes n}.
$$
Then a Yang-Baxter operator is such that $T_iT_j=T_jT_i$ when $|i-j|\geq2$ and $T_iT_{i+1}T_i=T_{i+1}T_iT_{i+1}$.
For each $n$, the $T$-symmetrizator is defined as
$P_T^{(n)}:\ch^{\otimes n}\rightarrow \ch^{\otimes n}$,
where $P_T^{(1)}:=I$, $P_T^{(2)}:= I + T_1$ and, for $n\geq 2$, is recursively given by
\begin{equation}
\label{Pn}
P_T^{(n+1)}:=(I \otimes P_T^{(n)})R^{(n+1)}=(R^{(n+1)})^*(I \otimes P_T^{(n)})\,\,,
\end{equation}
where $R^{(n)}:\ch^{\otimes n}\rightarrow \ch^{\otimes n}$ is such that
\begin{equation}
\label{Rn}
R^{(n)}:=I + T_1 + T_1T_2 +\ldots + T_1T_2 \cdots T_{n-1}\,.
\end{equation}
We notice that, if $P_T^{(n)}=0$ then $P_T^{(n+m)}=0$, for $m\geq 0$.
Recall that an operator $V:\ch\rightarrow \ch$ is called a Hecke operator if there exists $q\geq -1$ such that
\begin{equation}
\label{he}
V^2=(q-1)V + qI\,\,.
\end{equation}
A Yang-Baxter operator satisfying \eqref{he} is called a Yang-Baxter-Hecke one. Notice that, in such a case, $T$ is necessarily bounded with $\|T\|= \max(1,|q|)$.
From now on, we get $T$ a selfadjoint Yang-Baxter-Hecke operator.

In Proposition 1 of  \cite{Boz}, it was proven that, for each $n$, $P_T^{(n)}$ is similar to a selfadjoint projection, namely
\begin{equation}
\label{Psquare}
(P_T^{(n)})^2=\underline{n}!P_T^{(n)}=\underline{n}!(P_T^{(n)})^* \geq 0\,,
\end{equation}
where $\underline{n}:=1+q+q^2+ \ldots + q^{n-1}$, $\underline{n}!:=\underline{1}\cdot \underline{2}\cdots \underline{n}$ and moreover,
\begin{equation}
\label{pnorm}
\|P_T^{(n)}\|=\underline{n}!\,.
\end{equation}
As a consequence, one can define a pre-inner product, called $T$-deformed, given for $\xi\in \ch^{\otimes n}$, $\eta\in \ch^{\otimes m}$, by
\begin{equation}
\label{scpro}
\langle\xi,\eta\rangle_T:=\delta_{n, m}\langle\xi,P_T^{(n)}\eta\rangle\,\,,
\end{equation}
where in the r.h.s. one finds the usual scalar product in the full Fock space $\bigoplus_{n=0}^\infty \ch^{\otimes n}$, and $\ch^{\otimes 0}=\mathbb{C}$.
For each $n$, after dividing out by the kernel of $\langle\cdot,\cdot\rangle_T$, denote by $\ch_T^n$ the Hilbert space completion of $\ch^{\otimes n}$ w.r.t. the scalar product \eqref{scpro}. The $T$-deformed Fock space $\cf_T(\ch)$ is then defined as $\bigoplus_{n=0}^\infty \ch_T^n$, with $\ch_T^0=\mathbb{C}$ and $\ch_T^1=\ch$.
The vector $\Omega:=(1,0,0,\ldots)$ is called the vacuum and $\ch_T^n$ is said the $n$-{\it particles space}. The $T$-vacuum expectation, denoted directly as {\it the Fock vacuum} is the vector state induced by $\Om$.

For each $f\in \ch$, $n\in\mathbb{N}$, the creation operator (written without the subscript "$T$" for not to load the terminology) is given by
$$
a^\dagger(f)\xi:=f\otimes \xi\,,\,\,\, \xi\in\ch^{\otimes n}/\ker(\langle\cdot,\cdot\rangle_T) \,.
$$
By \eqref{Pn}, it is well defined from $\ch^{\otimes n}/\ker(\langle\cdot,\cdot\rangle_T)$ on $\ch^{\otimes (n+1)}/\ker(\langle\cdot,\cdot\rangle_T)$. The boundedness of $P^{(n)}$ ensures that $a^\dagger(f)$ can be extended as bounded operator from $\ch_T^n$ into $\ch_T^{(n+1)}$.
The annihilation operator $a(f)$ is defined as the conjugate of $a^\dagger(f)$ w.r.t. $\langle\cdot,\cdot\rangle_T$ with the property $a(f)\Omega=0$.
Since the total set of the {\it finite particle vectors}
$$
\cf^0_T(\ch):=\left\{\sum_{n=0}^N c_n\xi_n \mid N\in\mathbb{N}, c_n\in\mathbb{C}, G_n\in\ch_T^n, n=0,1,2,\ldots \right\}
$$
is included in both the domains of creator and annihilator and it is invariant under their action, it yields a dense domain in $\cf_T(\ch)$ where such operators are formally adjoint each other.
Finally, it is easily seen that, for $f\in \ch$,
\begin{equation}
\label{creator}
\big\|a(f)\lceil_{\ch_T^n}\big\|_T=\big\|a^\dagger(f)\lceil_{\ch_T^n}\big\|_T\leq \big\|R^{(n+1)}\big\|^{\frac{1}{2}}\big\|f\big\|\,,
\end{equation}
where the norms without the subscript "$T$" correspond to the free case (i.e. on the so called full Fock space).

Now we pass to the description of the analogue of the Bogoliubov automorphisms for the Yang-Baxter-Hecke quantisation.
\begin{prop}
\label{rash}
Let $U\in\cb(\ch)$ be a unitary operator such that
\begin{equation}
\label{commcz}
[T,U\otimes U]=0\,.
\end{equation}
Then $U\Om=\Om$, $\underbrace{U\otimes \cdots \otimes U}_{n\,\, \text{times}}$ acting on $\ch^{\otimes n}$, $n=1,2,\dots$, uniquely defines
a unitary operator $\cf_T(U)\in\cb(\cf_T(\ch))$ satisfying
$$
\cf_T(U)a(f)\cf_T(U^*)=a(Uf)\,,\quad f\in\ch\,.
$$
\end{prop}
\begin{proof}
Since \eqref{commcz}, one gets that for any $n$, the $n$-th tensor products of $U$ respects the kernel of $T$, so they are well defined on the relative quotient space $\ch^{\otimes n}_T$, giving a unitary $\cf_T(U)$ acting on $\cf_T(\ch)$. This allows to reach the thesis using the well known fact that it holds true in the free case (i.e. $T=0$). The details are left to the reader.
\end{proof}
\noindent
Such an automorphism unitarily implemented as in Proposition \ref{rash} is simply called a {\it Bogoliubov automorphism}.

Fix an Hilbert space $\ch$, and consider the $T$-deformed Fock space $\cf_T(\ch)$, where $T$ is a selfadjoint Yang-Baxter-Hecke operator on $\ch$. Let us take without loss of generality, $\ch= l^2(J)$ for some index set $J$ with cardinality the Hilbertian dimension of $\ch$.
If $e_j$, $j\in J$, is the generic element of the canonical basis, we use the following notations
$$
a_j:=a(e_j)\,\,, \,\,\,\,a^\dagger_j:=a^\dagger(e_j)\,\,.
$$
After defining the matrix $t$ (see \cite{BS}) by
$$
T(e_i\otimes e_j):=\sum_{k,l\in J}t_{ij}^{kl} e_k\otimes e_l,
$$
one obtains the following commutation rule for creation and annihilations operators (cf. \cite{BS}):
\begin{equation}
\label{commrul}
a_ia^\dagger_j - \sum_{k,l\in J}t_{jl}^{ik} a^\dagger_k a_l=\delta_{ij}I\,,\,\,\,\, i,j\in J.
\end{equation}
Because of the possibly infinite sum in \eqref{commrul}, such a commutation rule is a-priori only formal, even on the finite particle vectors $\cf^0_T(\ch)$, where every summand is always well defined by \eqref{creator}.
The next results shows it is meaningful in all the situations considered in the present paper.
\begin{prop}
Fix a vector $\xi\in\ch^{\otimes n}$. For each finite subset $J_0\subset J$, we get
$$
\left\|\sum_{k,l\in J_0}t_{jl}^{ik} a^\dagger_k a_l\xi\right\|_T\leq\|T\|\|R^{(n)}\|\|\xi\|_T\,.
$$
\end{prop}
\begin{proof}
Define $\G_{ij}:=\langle\,{\bf\cdot}\,e_i,e_j\rangle\otimes\id$ (the partial integration in the first variable w.r.t. the matrix elements $\langle\,{\bf\cdot}\,e_i,e_j\rangle$). By using the free annihilator (cf. \cite{BS}) $l(f)$ satisfying $a(f)=l(f)R^{(n)}$ on the $n$-particle subspace $\ch^{\otimes n}$, and taking into account that the norms, as well as the inner products without suffix are referred to the free ones, it is straightforward to show by a repeated application of \eqref{Psquare}, and \eqref{Pn},
\begin{align*}
&\big\|\sum_{k,l\in J_0}t_{jl}^{ik} a^\dagger_k a_l\xi\big\|_T^2
=\big\langle P^{(n)}_T (\G_{ji}(T)\otimes I)R^{(n)}\xi,(\G_{ji}(T)\otimes I)R^{(n)}\xi\big\rangle\\
=&1/\underline{n}!\big\langle (\G_{ji}(T)^*\otimes I)R^{(n)}(R^{(n)})^*(\G_{ji}(T)\otimes I)P^{(n)}_T\xi,P^{(n)}_T\xi\big\rangle\\
\leq&1/\underline{n}!\|T\|^2\|R^{(n)}\|^2\langle(P^{(n)}_T)^2\xi,\xi\rangle=\|T\|^2\|R^{(n)}\|^2\|\xi\|^2_T\,.
\end{align*}
\end{proof}
\noindent
If the sum in \eqref{commrul} is finite, one can express any word in the $a, a^\dagger$ in the so called {\it Wick form} (from the terminology introduced in Quantum Field Theory by the Italian theoretical physicist Gian Carlo Wick), that is when all the creators appears on the left w.r.t. all the annihilators.

Concerning details on the various quantisation schemes, we refer the reader to \cite{Boz, BR1, BS} for the proofs and further details

\section{ergodic properties for dynamical systems arising from Yang-Baxter-Hecke quantisation}
\label{sec3}

It is our aim to study some ergodic properties for the dynamical systems based on the shift naturally acting on some concrete $C^*$-algebras generated by creation and annihilation operators on $\cf_T(\ch)$. In so doing, we meet two obstructions. The first one is that, from \eqref{creator} it follows that $a_i$ and $a^\dagger_i$ are not necessarily bounded on
$\cf^0_T(\ch)$. The simplest well known example is the Canonical Commutation Relation (CCR for short) algebra describing physical particles obeying to the Bose statistics. We will see below in Subsection \ref{ecfi} that we can indeed consider (concrete) suitable $C^*$-dynamical systems enjoying very strong ergodic properties, even in this situation. Another obstruction concerns the fact that the sum appearing in \eqref{commrul} might be infinite. In this case we cannot express any words in annihilators and creators in the Wick order. This is indeed the case of the Boolean (cf. \cite{CrFid, Fbo}), as well as the Monotone $C^*$-algebras. Yet, we can reach or program even for these stochastic processes, as we will see below.

In order to pursue our goal, we need to make another condition, namely
\begin{equation}
\label{unifbound}
M_T:=\sup_{n\in\mathbb{N}}\|R^{(n)}\|<\infty\,.
\end{equation}
In this case, by means of \eqref{creator} one has that, for each $f\in\ch$
\begin{equation*}
\|a(f)\|_T=\|a^\dagger(f)\|_T\leq\sqrt{M_T}\|f\|\,.
\end{equation*}
The uniform boundedness \eqref{unifbound} is satisfied in many cases of interests, either when the Yang-Baxter $T$ is of Hecke type or when it is a strict contraction. Among the latter family, we mention the $q$-deformed cases, $q\in(-1,1)$. Moreover we will see below that, within the Hecke class, it is satisfied also for the Boolean and Monotone cases.
It is worth noticing that \eqref{unifbound} is only sufficient for the boundedness of the annihilators. In fact,
one can exhibit examples (i.e. the CAR algebra) in which creators and annihilators are bounded but the operators $R^{(n)}$ are not uniformly bounded. From now on, we deal only with cases for which $M_T<+\infty$, if it is not otherwise specified.

From now on, we put $\ch=\ell^2(\bz)$ and consider the unital selfadjoint algebra $\gar^0_T$ generated by $\{a_i\mid i\in\mathbb{Z}\}$ acting on $\cf_T(\ch)$. Its norm closure $\gar_T$ is the concrete $C^*$-algebra associated to the $T$-quantisation. We also consider the unital $C^*$-subalgebra $\gg_T$ generated by the selfadjoint part of the
annihilators $\{a_i+a^\dagger_i\mid i\in\mathbb{Z}\}$.

Let $\{U_g\mid g\in G\}\subset\cb(\ch)$ be a unitary representation of the group $G$ such that
$$
[T,U_g\otimes U_g]=0\,,\quad g\in G\,.
$$
From Proposition \ref{commcz}, one has that
$$
g\in G\mapsto\cf_T(U_g)\in\cb(\cf_T(\ch))
$$
defines a unitary representation of $G$ on $\cf_T(\ch)$, and
$$
\a_g(A):=\cf_T(U_g)A\cf_T(U_{g^{-1}})\,,\quad g\in G\,, A\in\gar_T
$$
is an action of $G$ on $\gar_T$. Concerning the $C^*$-subalgebra $\gg_T$, the situation looks a bit more complicated as it is generated by the quantisation of the real part $\ell^2_\br(\bz)$. For this situation, the orthogonal group $O(\ell^2_\br(\bz))$ is naturally involved, see e.g. \cite{FM2} and the references cited therein. We decide not to pursue this aspect because in the cases treated in detail in the present paper (Monotone and Boolean), one has
$\gg_T=\gar_T$.

In the sequel, we suppose that the unitary $e_i\mapsto e_{i+1}$, $i\in\bz$ generating the right shift on $\ell^2(\bz)$, satisfies the condition of Proposition \ref{rash}, and then it acts as Bogoliubov automorphisms
$\a^n$, $n\in\bz$ on $\gar_T$ by
$$
\a(a_i):=a_{i+1}\,\,, i\in\mathbb{Z}
$$
on the generators. Then it acts also on $\gg_T$ by restriction. We denote by $\gar_T^\mathbb{Z}$ and $\gg_T^\mathbb{Z}$ the corresponding fixed point subalgebras of $\gar_T$ and $\gg_T$ respectively. By construction, the Fock vacuum expectation $\om=\langle\,\cdot\,\Om, \Om \rangle$ is invariant for $\a$.

Here, we want to determine some conditions which make easier to understand if the $C^*$-dynamical systems $(\gar_T, \alpha)$ and $(\gg_T, \alpha)$ enjoy the strong ergodic property like unique ergodicity or unique mixing.
For a Yang-Baxter operator $T$ with norm strictly less than $1$, it was proven in \cite{DykFid} that these $C^*$-algebras are uniquely mixing for the shift with the the vacuum expectation as the only invariant state. If one considers a Yang-Baxter-Hecke operator, this result cannot be directly applied since $\|T\|\geq 1$. Indeed, we will show that condition \eqref{unifbound} is sufficient to achieve partially the above mentioned strong mixing property.
The following technical results will be useful for this purpose.
\begin{lem}
\label{sum1}
For a Yang-Baxter-Hecke selfadjoint operator $T$ on $\ch= \ell^2(\mathbb{Z})$ for which \eqref{commcz} holds, if $\{\xi_i\}^n_{i=1}\subset \ch_T^{k}$ and $\{f_i\}^n_{i=1}$ is an orthonormal set of $\ch$, then
$$
\left\|\sum^n_{i=1} a^\dagger(f_i)\xi_i \right\|_T \leq \sqrt{nM_T} \max_{i=1,\ldots, n} \|\xi_i\|_T\,.
$$
\end{lem}
\begin{proof}
Since \eqref{unifbound}, by \eqref{Pn}, \eqref{Psquare} and the positivity of $P^{(n)}_T$ operators, one has for each $k$,
$$
(P^{(k+1)}_T)^2=(I\otimes P^{(k)}_T)R^{(k+1)}_T(R^{(k+1)}_T)^*(I\otimes P^{(k)}_T)\leq M_T^2(I\otimes P^{(k)}_T)^2\,,
$$
which immediately leads to $P^{(k+1)}_T\leq M_T(1\otimes P^{(k)}_T)$ as the square root is operator monotone.
From now on the proof follows, up to slight modifications, the same ideas of Lemma 3.1 in \cite{DykFid}. We report it here for the convenience of the reader.
\begin{align*}
&\big\langle\sum_{j=1}^n a^\dagger(f_j)\xi_j, \sum_{j=1}^n a^\dagger(f_j)\xi_j\big\rangle_T = \big\langle P^{(k+1)}_T \sum_{j=1}^n a^\dagger(f_j)\xi_j, \sum_{j=1}^n a^\dagger(f_j)\xi_j\big\rangle \\
\leq & M_T \big\langle(I\otimes P^{(k)}_T)\sum_{i=1}^n f_j\otimes \xi_j, \sum_{i=1}^n f_j\otimes \xi_j \big\rangle = M_T\sum_{i,j=1}^n \big\langle f_i\otimes P^{(k)}_T \xi_i, f_j\otimes \xi_j \big\rangle \\
= & M_T\sum_{i,j=1}^n \langle f_i, f_j\rangle \big\langle P^{(k)}_T \xi_i, \xi_j \big\rangle = M_T\sum_{i,j=1}^n \langle f_i, f_j\rangle \langle \xi_i, \xi_j\rangle_T \\
= & M_T\sum_{i=1}^n \langle \xi_i, \xi_i\rangle_T \leq nM_T \max_{i=1,\ldots, n} \|\xi_i\|_T^2\,.
\end{align*}
\end{proof}
\begin{prop}
\label{sum2}
Let $T$ a Yang-Baxter-Hecke selfadjoint operator on $\ch=\ell^2(\mathbb{Z})$ satisfying \eqref{unifbound}, \eqref{commcz}. If $0\leq k_1<k_2< \ldots < k_n$ are natural numbers and $e_{i_1},e_{i_2}, \ldots, e_{i_r}$ are vectors of the canonical basis of $\ell^2(\mathbb{Z})$, then
$$
\left \|\sum_{h=1}^n \a^{k_h}(a^\dagger(e_{i_1})a^\sharp(e_{i_2})\cdots a^\sharp(e_{i_r}))\right \|_T\leq \sqrt {n(M_T)^{r}}\,,
$$
$$
\left \|\sum_{h=1}^n \a^{k_h}(a^\sharp(e_{i_1})a^\sharp(e_{i_2})\cdots a(e_{i_r}))\right \|_T\leq \sqrt {n(M_T)^{r}}\,,
$$
where $\sharp\in\{1,\dagger\}$.
\end{prop}
\begin{proof}
Suppose $j$ is the number of annihilators in the sequences
$$
a^\dagger(e_{i_1})a^\sharp(e_{i_2})\cdots a^\sharp(e_{i_r})
$$
and
$$
a^\sharp(e_{i_1})a^\sharp(e_{i_2})\cdots a(e_{i_r})\,.
$$
Fix $\xi$ a unit vector in $\ch_T^{m}$, $m\geq j$ and denote
$$
\xi_h:=a^\sharp(e_{i_2+k_h})\cdots a^\sharp(e_{i_r+k_h})\xi\,.
$$
Since \eqref{creator}, one has $\|\xi_h\|_T \leq \sqrt{(M_T)^{r-1}}$. Now we consider
\begin{equation}
\label{sharp}
\left \|\sum_{h=1}^n \a^{k_h}(a^\sharp(e_{i_1})a^\sharp(e_{i_2})\cdots a(e^\sharp_{i_r}))\right \|_T\,.
\end{equation}
We first suppose $a^\sharp(e_{i_1})=a^\dagger(e_{i_1})$ in \eqref{sharp}. Then, taking into account that $\langle e_{i_i+k_h},e_{i_i+k_{\widehat{h}}}\rangle=\delta_{h,\widehat{h}}$, $\delta_{h,\widehat{h}}$ being the Kronecker symbol, one has
\begin{align*}
&\left \|\sum_{h=1}^n \a^{k_h}(a^\dagger(e_{i_1})a^\sharp(e_{i_2})\cdots a^\sharp(e_{i_r}))\xi\right \|_T \\
= &\left \|\sum_{h=1}^n a^\dagger(e_{i_1+k_h})a^\sharp(e_{i_2+k_h})\cdots a^\sharp(e_{i_r+k_h})\xi\right \|_T \\
= &\left \|\sum_{h=1}^n a^\dagger(e_{i_1+k_h})\xi_h \right \|_T\leq\sqrt {n(M_T)^{r}}\,,
\end{align*}
where the last inequality follows from Lemma \ref{sum1}.
If instead, $a^\sharp(e_{i_r})=a(e_{i_r})$ in \eqref{sharp}, the desired inequality follows from the first part by taking the adjoint.
\end{proof}
\noindent
Here, there is the main result of the present section concerning the ergodic properties of $(\gar_T, \alpha)$.
\begin{thm}
\label{mixing}
Let $T$ be a Yang-Baxter-Hecke selfadjoint operator on $\ch= \ell^2(\mathbb{Z})$. Suppose that the sum in \eqref{commrul} is finite, and furthermore \eqref{unifbound}, \eqref{commcz} hold true. Then the dynamical system $(\gar_T, \alpha)$ is  uniquely mixing with $\om$ the unique invariant state.
\end{thm}
\begin{proof}
For $X\in\gar^0_T$, put $E(X):=\om(X)I$. Proposition \ref{sum2} easily implies that
\begin{equation}
\label{ergo}
\lim_n\frac{1}{n}\sum_{k=1}^{n}\alpha^{l_k}(X)=E(X)\,,\quad X\in\gar^0_T\,,
\end{equation}
for each fixed subsequence $\{l_1,\dots,l_k,\dots\}\subset\bn$, where the limit is understood in norm. A standard 3-$\eps$ argument implies that \eqref{ergo} holds true for a generic element of $\gar_T$ as well. The proof follows from Proposition 2.3 in \cite{F22}.
\end{proof}
\noindent
Theorem \ref{mixing} holds true even for the $C^*$-dynamical system $(\gg_T,\a)$ by restriction. In addition, it holds true for all the cases listed in \cite{Boz} which fulfil the hypothesis requested in the above statement, for the forward shift $\a$, and for the backward one corresponding to the automorphism $\a^{-1}$ as well.

\section{shift invariant states of the monotone $C^*$-algebra}
\label{monsec}

In the following lines, we will see as some of the general results contained in Section \ref{sec3} can be directly applied to study the set of the shift invariant states on
the Monotone $C^*$-algebra, and their ergodic properties. For such a purpose, we explicitly describe some basic facts on Monotone Fock space and fundamental operators acting on them, see for more details \cite{CrLu, Lu2, Lu, M}.

For $k\geq 1$, denote $I_k:=\{(i_1,i_2,\ldots,i_k) \mid i_1< i_2 < \cdots <i_k, i_j\in \mathbb{Z}\}$, and for $k=0$, we take $I_0:=\{\emptyset\}$, $\emptyset$ being the empty sequence. The Hilbert space $\ch_k:=\ell^2(I_k)$ is precisely the $k$-particles space for the Monotone quantisation. In particular, the $0$-particle space $\ch_0=\ell^2(\emptyset)$ is identified with the complex scalar field $\mathbb{C}$. The Monotone Fock space is precisely $\cf_m=\bigoplus_{k=0}^{\infty} \ch_k$.

Given an increasing sequence $\a=(i_1,i_2,\ldots,i_k)$ of natural numbers, we denote by $e_{(\a)}$ the generic element of canonical basis of $\cf_m$. For each pair of such sequences
$\a=(i_1,i_2,\ldots,i_k)$, $\b=(j_1,j_2,\ldots,j_l)$, we say $\a < \b$ if $i_k < j_1$. By convention, $I_0<\a$ for each $\a\neq I_0$.
The Monotone creation and annihilation operators are respectively given, for any $i\in \mathbb{Z}$, by
\begin{equation*}
a^\dagger_ie_{(i_1,i_2,\ldots,i_k)}:=\left\{
\begin{array}{ll}
e_{(i,i_1,i_2,\ldots,i_k)} & \text{if}\, (i)< (i_1,i_2,\ldots,i_k)\,, \\
0 & \text{otherwise}\,, \\
\end{array}
\right.
\end{equation*}
\begin{equation*}
a_ie_{(i_1,i_2,\ldots,i_k)}:=\left\{
\begin{array}{ll}
e_{(i_2,\ldots,i_k)} & \text{if}\, k\geq 1\,\,\,\,\,\, \text{and}\,\,\,\,\,\, i=i_1\,,\\
0 & \text{otherwise}\,. \\
\end{array}
\right.
\end{equation*}
By dropping the subscript "$m$" for the norms, one can check that $\|a^\dagger_i\|=\|a_i\|=1$, see e.g. \cite{Boz}, Proposition 8.
Moreover, $a^\dagger_i$ and $a_i$ are mutually adjoint and satisfy the following relations
\begin{equation}
\label{comrul}
\begin{array}{ll}
  a^\dagger_ia^\dagger_j=a_ja_i=0 & \text{if}\,\, i\geq j\,, \\
  a_ia^\dagger_j=0 & \text{if}\,\, i\neq j\,.
\end{array}
\end{equation}
In addition, the following commutation relation
\begin{equation*}
a_ia^\dagger_i=I-\sum_{k\leq i}a^\dagger_k a_k
\end{equation*}
is also satisfied. The latter is indeed of the type of those considered in \eqref{commrul}. In fact, the Monotone corresponds to the Yang-Baxter-Hecke operator $T_m:\ch_1\otimes \ch_1 \rightarrow \ch_1\otimes \ch_1$ given for $i,j\in\mathbb{N}$,
\begin{equation}
\label{TM}
T_m(e_i\otimes e_j):=\left\{
\begin{array}{ll}
                       0 & \text{if}\,\, i<j\,,\\
                       -(e_i\otimes e_j) & \text{if}\,\, i\geq j\,.
                     \end{array}
                     \right.
\end{equation}
Indeed, by Proposition 7 of \cite{Boz}, the $T_m$-deformed Fock space coincides with the Monotone Fock space. Here, $q=0$ and $\underline{n}!=1$ for all $n$, whereas $P_m^{(n)}$ is the orthogonal projection of $\ch^{\otimes n}$ onto the linear span of $\{e_{i_1}\otimes e_{i_2}\otimes \cdots \otimes e_{i_n} | i_1<i_2< \cdots < i_n\}$.

The next result ensures in addition that operators which give rise to the Monotone deformation of the full scalar product realise the fundamental condition \eqref{unifbound}.
\begin{lem}
\label{RN}
Let $T_m$ be defined as in \eqref{TM}. Then for each $n\in\mathbb{N}$, one has $0\leq R^{(n)}\leq I $.
\end{lem}
\begin{proof}
Notice that $R^{(1)}=I$, and $R^{(2)}=I + T_m$ is the orthogonal projection onto the linear space generated by the vectors $e_{i_1} \otimes e_{i_2}$ with $i_1 < i_2$.
Moreover, fix $n$ and take $i_1,i_2,\ldots, i_n \in\mathbb{Z}$. Using \eqref{TM}, we have for $i_1\geq i_2 \geq \ldots \geq i_n$,
\begin{equation*}
R^{(n)}(e_{i_1} \otimes \cdots \otimes e_{i_n})=\left\{
        \begin{array}{ll}
          e_{i_1} \otimes \cdots \otimes e_{i_n} & \text{if}\,\, n\,\, \text{ is odd}\,,\\
          0 & \text{if}\,\,  n\,\, \text{is even}\,,
        \end{array}
        \right.
\end{equation*}
and if $j= \min \{1,\ldots, n-1\}$ s.t. $i_j < i_{j+1}$,
$$
 R^{(n)}(e_{i_1} \otimes \cdots \otimes e_{i_n})=\big(R^{(j)}\otimes I \otimes \cdots \otimes I\big)(e_{i_1} \otimes \cdots \otimes e_{i_n})\,.
$$
The thesis follows by induction.
\end{proof}
\noindent
The $C^*$-algebra $\gar_m$ and its subalgebra $\gg_m$ acting on $\cf_m$, are the unital
$C^*$-algebras generated by the annihilators $\{a_i\mid i\in\mathbb{Z}\}$, and the selfadjoint part of
annihilators $\{s_i\mid i\in\mathbb{Z}\}$ respectively, with $
s_i:=a_i+a^+_i$. We will see later (cf. Proposition \ref{invmon2}) that these two algebras coincide.

Notice that \eqref{TM} ensures that \eqref{commcz} is satisfied when $U$ is the unitary shift on $\ell^2(\mathbb{Z})$. Then from Proposition \ref{rash}, the right shift $\alpha$ acts on $\gar_m$ in the usual way:
$$
\alpha(a_i):=a_{i+1}\,\,, i\in\mathbb{Z}\,.
$$
From now on, $P_\Om$ will denote the orthogonal projection onto the linear space generated by the vacuum vector $\Om$.

We first notice that $(\gar_m, \alpha)$ cannot be uniquely ergodic (w.r.t. the fixed point algebra).
Indeed, let us fix $i\in \mathbb{Z}$. As for $n\to+\infty$, $\alpha^n(a_ia^\dagger_i)\downarrow P_\Om$ (see e.g. the proof of Proposition \ref{vncz} below), we get
$$
\lim_{n\rightarrow \infty} \frac{1}{n} \sum_{k=0}^{n-1}\alpha^k(a_ia^\dagger_i)=P_\Om
$$
in the strong operator topology. On the other hand,
the convergence is not in norm, since
$$
1\geq\bigg\|\frac{1}{n} \sum_{k=0}^{n-1}\alpha^k(a_ia^\dagger_i)-P_\Om \bigg\|\geq \bigg\|\bigg(\frac{1}{n} \sum_{k=0}^{n-1}\alpha^k(a_ia^\dagger_i)-P_\Om\bigg)e_{(i+n)} \bigg\|= 1\,.
$$
\vskip.3cm
\begin{rem}
This fact performs an example for the general failure of the thesis exposed in Theorem \ref{mixing} when the sum \eqref{commrul} is not finite, even if all the other conditions are satisfied. Conversely, for the Boolean case (cf. Section \ref{boolsec}) we are able to see that $(\gar_\text{Boole},\a)$ is indeed uniquely mixing w.r.t. the fixed point algebra, even if
the sum \eqref{commrul} is not finite.
\end{rem}
\vskip.3cm
\noindent
The goal of the present section is to study the convex set of the shift invariant states $\cs_{\bz}(\gar_m)$ on the Monotone algebra $\gar_m$. 
We start by describing in details the unital $*$-algebra $\gar_m^0$ generated by the Monotone commutation relations acting on the Monotone Fock space.
\vskip.3cm
\begin{defin}
\label{deflampi}
A word $X$ in $\gar^0_m$ is said to have a $\lambda$-\textbf{form} if there are $m,n\in\left\{  0,1,2,\ldots
\right\}$ and $i_1<i_2<\cdots < i_m, j_1>j_2> \cdots > j_n$ such
that
$$
X=a_{i_1}^{\dagger}\cdots a_{i_m}^{\dagger} a_{j_1}\cdots a_{j_n}\,,
$$
with $X=I$, the empty word, if $m=n=0$. Its length is $l(X)=m+n$.
Furthermore, $X$ is said to have a $\pi$-\textbf{form} if there are $m,n\in\left\{0,1,2,\ldots
\right\}$, $k\in\mathbb{Z}$, $i_1<i_2<\cdots < i_m, j_1>j_2> \cdots > j_n$ such that $i_m<k>j_1$ and
$$
X=a_{i_1}^{\dagger}\cdots a_{i_m}^{\dagger} a_{k}a_{k}^{\dagger} a_{j_1}\cdots a_{j_n}\,.
$$
Its length is $l(X)=m+2+n$.
\end{defin}
\vskip.3cm
\begin{lem}
\label{lampi1}
For each $a_j,a_k\in \gar_m$, one has
\begin{equation}
\label{b-le01a}
a_{k}a_{j}a_{j}^{\dagger}=\delta_{k)}\left(  j\right)a_{k}\,,\quad a_{j}a_{j}^{\dagger}a_{k}^{\dagger}=\delta_{k)}\left(j\right)a_{k}^{\dagger}.
\end{equation}
In addition, if $j\leq k$,
\begin{equation}
\label{b-01b}
a_{j}a_{j}^{\dagger}a_{k}=a_{k}\,,\quad a_{k}^{\dagger}a_{j}a_{j}^{\dagger}=a_{k}^{\dagger},
\end{equation}
where
\begin{equation*}
\delta_{k)}\left(j\right):=\left\{
\begin{array}{ll}
1 & \text{if } j<k\,,\\
0 & {\rm otherwise}\,.
\end{array}
\right.
\end{equation*}
\end{lem}
\begin{proof}
By the definition of Monotone creation and annihilation operators, one easily obtains \eqref{b-le01a}.
Concerning \eqref{b-01b}, we get
$$
a_{j}a_{j}^{\dagger}a_{k}\Omega=0=a_{k}\Omega\,,
$$
$$
a_{j}a_{j}^{\dagger}a_{k}e_{h}=\delta_{k,h}a_{j}a_{j}^{\dagger}\Omega=\delta_{k,h}\Omega=a_{k}e_{h}\,,\quad h\in\mathbb{Z}\,.
$$
If moreover, $h_{1}<h_{2}$ and $\xi\in \cf_m$,
\begin{align*}
a_{j}a_{j}^{\dagger}a_{k}\left(e_{h_{1}}\otimes e_{h_{2}}\otimes\xi\right)   &
=\delta_{k,h_{1}}a_{j}a_{j}^{\dagger}\left(e_{h_{2}}\otimes\xi\right)\\
&=\delta_{k,h_{1}}\delta_{h_{2})}\left(  j\right)  \left(  e_{h_{2}}\otimes
\xi\right) \\
&=a_{k}\left(e_{h_{1}}\otimes e_{h_{2}}\otimes\xi\right)
\end{align*}
where, in the last line we used the identity $\delta_{k,h_{1}}\delta_{h_{2})}\left(j\right)  =\delta_{k,h_{1}}$, which holds since $j\leq k$ and $h_{1}<h_{2}$.
The second formula of (\ref{b-01b}) is achieved form the first one by taking the adjoint.
\end{proof}
\begin{lem}
\label{lampi2}
Any nonnull element $X$ in $\gar_m^0$ is a finite linear combinations of $\lambda$-forms or
$\pi$-forms.
\end{lem}
\begin{proof}
Let $X=a_{i_i}^{\sharp}\cdots a_{i_n}^{\sharp}$ with $n\in\mathbb{N},$ $i_1,\ldots, i_n\in\mathbb{Z}$ and $\sharp\in\{1,\dagger\}$. If $n=1,2$, one easily achieves the result. Now we prove it by induction on the length $l(X)$.

Let us denote $X=X_1a_{i_n}^{\sharp}$, where $X_1=a_{i_1}^{\sharp}\cdots a_{i_{n-1}}^{\sharp}$. If $X$ is non null, then $X_1$ is so, and consequently by induction, $X_1$ has the $\lambda$-form or the $\pi$-form. We also suppose that $l(X_1)\geq2$.

We first assume that $X_1$ has a $\lambda$-form, i.e. $X_1=a_{i_1}^{\dagger}\cdots a_{i_m}^{\dagger} a_{j_1}\cdots a_{j_k}$ and $i_1<\cdots < i_m$, $j_1>\cdots > j_k$.
If $k>0$ and take $a^\sharp_{i_n}=a_{i_n}$, from \eqref{comrul} one has $j_k> i_n$, i.e. $X$ has a $\lambda$-form. If instead, $k>0$ and take $a^\sharp_{i_n}=a_{i_n}^{\dagger}$, again \eqref{comrul}, \eqref{b-le01a} and \eqref{b-01b} give a $\lambda$-form for $X$. Namely, $X=a_{i_1}^{\dagger}\cdots a_{i_m}^{\dagger} a_{j_1}\cdots a_{j_{k-1}}$ (in fact, in this case necessarily $j_k=i_n$ since $X$ is non null). Suppose that $k=0$, then according to our assumptions, $m\geq 2$. After taking $a_{i_n}^{\sharp}=a_{i_n}$, we get $X$ has a $\lambda$-form, whereas when we consider $a_{i_n}^{\sharp}=a_{i_n}^{\dagger}$, from \eqref{comrul} one has $i_1<i_2<\cdots < i_m< i_n$, then giving also a $\lambda$-form to $X$.

Now we assume $X_1$ has a $\pi$-form, i.e. $X_1=a_{i_1}^{\dagger}\cdots a_{i_m}^{\dagger}a_pa^\dagger_p a_{j_1}\cdots a_{j_k}$, $i_1<\cdots < i_m$, $j_1>\cdots > j_k$ and $i_m<p>j_1$. If $k>0$, one uses the same arguments developed above to achieve a $\pi$-form for $X$.
If instead, $k=0$ and get $a_{i_n}^{\sharp}=a_{i_n}$, we have to discuss the cases $p\leq i_n$ and $p> i_n$, respectively. In the former, from \eqref{b-01b} one has $X=a_{i_1}^{\dagger}\cdots a_{i_m}^{\dagger}a_{i_n}$, i.e. $X$ has a $\lambda$-form. When, instead $p> i_n$, $X$ has a $\pi$-form.
The last case can occur when $k=0$ and $a_{i_n}^{\sharp}=a_{i_n}^\dagger$. Here, from \eqref{b-le01a} it follows $X$ does not vanish only if $p<i_n$, and in this circumstance $X=a_{i_1}^{\dagger}\cdots a_{i_m}^{\dagger}a_{i_n}^\dagger$, i.e. $X$ has a $\lambda$-form.

As each element in the $*$-algebra is a linear combination of $X$ as above, the thesis follows if one achieves the same representation for the left multiplication. Indeed, since the adjoint of any word in $\l$-form or
$\pi$-form is again a word in $\l$-form or $\pi$-form respectively, we get
$$
a_i^\sharp X=\big(X^*(a_i^\sharp)^*\big)^*=Y^*=Z\,,
$$
where $Y:=X^*(a_i^\sharp)^*$, and consequently $Z:=Y^*$ are words in $\l$-form or
$\pi$-form.
\end{proof}
\begin{lem}
\label{lampi3}
${}$\\
\begin{itemize}
\item[(i)] Two words in $\l$-form
$$
X=a^\dagger_{i_1}\cdots a^\dagger_{i_{m_1}}a_{j_1}\cdots a_{j_{n_1}}\,,\quad Y=a^\dagger_{k_1}\cdots a^\dagger_{k_{m_2}}a_{h_1}\cdots a_{h_{n_2}}
$$
are equal if and only if $m_1=m_2=m$, $n_1=n_2=n$ (with the convention that $m=n=0$ corresponds to the identity), and in addition, $i_l=k_l$, $l=1,2,\dots,m$, $j_l=h_l$, $l=1,2,\dots,n$.
\item[(ii)] Two words in $\pi$-form
$$
X=a^\dagger_{i_1}\cdots a^\dagger_{i_{m_1}}a_{r_1}a^\dagger_{r_1}a_{j_1}\cdots a_{j_{n_1}}\,, \quad
Y=a^\dagger_{k_1}\cdots a^\dagger_{k_{m_2}}a_{r_2}a^\dagger_{r_2}a_{h_1}\cdots a_{h_{n_2}}
$$
are equal if and only if $m_1=m_2=m$, $n_1=n_2=n$, and in addition, $i_l=k_l$, $l=1,2,\dots,m$, $r_1=r_2$, $j_l=h_l$, $l=1,2,\dots,n$.
\item[(iii)] If two words $X,Y$ of $\gar_m^0$ are equal, then those must be both of $\l$-form or of $\pi$-form.
\end{itemize}
\end{lem}
\begin{proof}
(i) We need to prove just the "only if" assertion. We suppose that $X,Y$ are both not coinciding with the empty word and $X=Y$. If $n_1< n_2$, one has
\begin{align*}
& a^\dagger_{i_1}\cdots a^\dagger_{i_{m_1}}a_{j_1}\cdots a_{j_{n_1}}(e_{j_{n_1}}\otimes \cdots \otimes e_{j_1})=\left\{
\begin{array}{ll}
e_{i_{1}}\otimes \cdots \otimes e_{i_{m_1}}\,, & \text{if } m_1>0\,, \\
\Om & \text{if } m_1=0\,,
\end{array}
\right.
\end{align*}
whereas
$$
a^\dagger_{k_1}\cdots a^\dagger_{k_{m_2}}a_{h_1}\cdots a_{h_{n_2}}(e_{j_{n_1}}\otimes \cdots \otimes e_{j_1})=0\,.
$$
Then, $n_{1}\geq n_{2}$. Similarly one can achieve $n_{2}\geq n_{1}$, and so $n_{1}=n_{2}$. By taking the adjoint, $m_{1}=m_{2}$ holds too. In this case,
\begin{align*}
&a^\dagger_{i_1}\cdots a^\dagger_{i_{m_1}}a_{j_1}\cdots a_{j_{n_1}}(e_{j_{n_1}}\otimes \cdots \otimes e_{j_1})\\
=& a^\dagger_{k_1}\cdots a^\dagger_{k_{m_1}}a_{h_1}\cdots a_{h_{n_1}}(e_{j_{n_1}}\otimes \cdots \otimes e_{j_1})\,,
\end{align*}
or equivalently,
$$
e_{i_{1}}\otimes \cdots \otimes e_{i_{m_1}}=\prod_{p=1}^{n_1}\delta_{h_p,j_p}e_{k_{1}}\otimes \cdots \otimes e_{k_{m_1}}\,.
$$
(ii) Here, we again need to prove just the "only if" assertion. So we suppose $X=Y$. The equalities $m_{1}=m_{2}$, $n_{1}=n_{2}$, $i_p=k_p,\ p\in \{1,\cdots, m_1\}$, $j_p=h_p, \ p\in \{1,\cdots, n_1\}$, can be realised as above. Now we prove $r_1=r_2$. Indeed, after recalling that the involved words do not vanish only if $\min\{r_{1},r_{2}\}> \max \{i_{m_1}, j_1\}$, when $r_1<r_2$,
\begin{align*}
&a^\dagger_{i_1}\cdots a^\dagger_{i_{m_1}}a_{r_1}a^\dagger_{r_1}a_{j_1}\cdots a_{j_{n_1}}(e_{j_{n_1}}\otimes \cdots \otimes e_{j_1}\otimes e_{r_2})\\
=& e_{i_{1}}\otimes \cdots \otimes e_{i_{m_1}}\otimes e_{r_2}
\end{align*}
and
\begin{align*}
& a^\dagger_{k_1}\cdots a^\dagger_{k_{m_1}}a_{r_2}a^\dagger_{r_2}a_{h_1}\cdots a_{h_{n_1}}(e_{j_{n_1}}\otimes \cdots \otimes e_{j_1}\otimes e_{r_2})=0\,.
\end{align*}
Similarly one cannot have $r_1>r_2$.

\noindent
(iii) We take $X$ in $\l$-form and $Y$ in $\pi$-form, that is
$$
X=a^\dagger_{i_1}\cdots a^\dagger_{i_{m_1}}a_{j_1}\cdots a_{j_{n_1}}\,,\quad
Y=a^\dagger_{k_1}\cdots a^\dagger_{k_{m_2}}a_{r_2}a^\dagger_{r_2}a_{h_1}\cdots a_{h_{n_2}}\,.
$$
Suppose without loss of generality, they are both reduced and prove that, necessarily, $X\neq Y$. Indeed, assume $X=Y$. Arguing as above, one gets such an equality cannot holds when $n_1\neq n_2$, or taking the adjoint, when $m_1\neq m_2$. Thus, from now on we assume $n:=n_{1}=n_{2}$ and $m:=m_{1}=m_{2}$. Since
\begin{align*}
&a^\dagger_{i_1}\cdots a^\dagger_{i_m}a_{j_1}\cdots a_{j_n}(e_{j_n}\otimes \cdots \otimes e_{j_1})=e_{i_{1}}\otimes \cdots \otimes e_{i_m}\,,
\end{align*}
and
\begin{align*}
& a^\dagger_{k_1}\cdots a^\dagger_{k_m}a_r a^\dagger_r a_{h_1}\cdots a_{h_n}(e_{j_n}\otimes \cdots \otimes e_{j_1})= \prod_{p=1}^n \delta_{h_p, j_p} e_{k_1}\otimes e_{k_m}\,,
\end{align*}
we need further to require $i_p=k_p,\ p\in \{1,\cdots, m\}$, and $j_p=h_p, \ p\in \{1,\cdots, n\}$. But in this case, for a given $p$ s.t. $\max\{i_m, j_1\}<p \leq r$,
$$
a^\dagger_{i_1}\cdots a^\dagger_{i_m}a_{j_1}\cdots a_{j_n}(e_{j_n}\otimes \cdots \otimes e_{j_1}\otimes e_p)=e_{i_{1}}\otimes \cdots \otimes e_{i_m}\otimes e_p\,,
$$
and
$$
a^\dagger_{i_1}\cdots a^\dagger_{i_m}a_ra^\dagger_ra_{j_1}\cdots a_{j_n}(e_{j_n}\otimes \cdots \otimes e_{j_1}\otimes e_p)=0\,,
$$
which contradicts $X=Y$.
\end{proof}
\vskip.3cm
\begin{rem}
Lemma \ref{lampi2} and Lemma \ref{lampi3} ensure that, first the words in $\l$-form and in $\pi$-form are reduced, and in addition they generate the whole $\gar_m^0$. It is possible to see (cf. Section 4.5 of \cite{CrFid1}) that they are not a basis for $\gar_m^0$. We again stress the identity
$I\in\cb(\cf_m)$ corresponds to the empty word (which is by definition of $\l$-form).
\end{rem}
\vskip.3cm
\noindent
Denote $\ga_m^0:=\spn\big\{X\in\gar_m^0\mid l(X)>0\big\}$.
It is a $*$-subalgebra of $\gar^0_m$.
We are going to show that $I$ does not belong to $\ga_m:=\overline{\ga_m^0}$.
It is worth of mention that, for each $i\in\mathbb{Z}$, $\sum_{k\leq i}a^\dagger_ka_k$ might not belong to $\ga_m$, since it is easy to check $\|a^\dagger_ka_k\|=1$ for any $k$, and so the series cannot be convergent in norm.

The following results are meaningful in the investigation of the relation between $\gar_m$ and $\ga_m$.
\begin{prop}
\label{ide}
For each $X\in\ga_m$ and $\a\in\bc$, one has $\|X+\a I\|\geq|\a|$, and consequently $I\notin\ga_m$.
\end{prop}
\begin{proof}
Let $X\in\ga_m^0$. It is a finite sum $X=\sum_{i}\beta_i X_i$, with each $X_i$ in $\l$-form or $\pi$-form and $l(X_i)>0$, that is $X_i=Y_ia_{j_i}^\sharp$. Take the unit vector $e_{n}\in\cf_m$ for any $n\in\bz$ with $n<\min{j_i}$. Obviously,
$a_{j_i}^\sharp e_{n}=0$ for each $j_i$. Then we get
\begin{equation*}
\|X+\a I\|\geq \|(X+\a I)e_{n}\|=|\a|\|e_{n}\|=|\a|\,.
\end{equation*}
For each $\eps>0$ and $X\in\ga_m$, we choose $X_\eps\in\ga_m^0$ such that $\|X-X_\eps\|<\eps$. By the triangle inequality and the above considerations, we get
\begin{align*}
|\a|\leq&\|X_\eps+\a I\|=\|(X_\eps-X)+(X+\a I)\|\\
\leq&\|X-X_\eps\|+\|X+\a I\|<\eps+\|X+\a I\|\,.
\end{align*}
The thesis follows as $\eps$ is arbitrary.
\end{proof}
\begin{prop}
\label{vncz}
$\ga_m$ is irreducible. Thus, it does not have an identity.
\end{prop}
\begin{proof}
Since $\Om$ is separating for $\ga'_m$, we only need to prove that $P_\Om$ belongs to $\ga''_m=\pi_\om(\gar_m)^{''}$.
In fact, $\sum_{k\in\mathbb{Z}}a^\dagger_ka_k\Om=0$, and for each $n\in\mathbb{N}$, $\xi:=e_{i_1}\otimes e_{i_2}\otimes \cdots \otimes e_{i_n}$, $i_1<i_2<\dots <i_n$, $i_j\in\mathbb{Z}$, one has $\sum_{k\in\mathbb{Z}}a^\dagger_ka_k\xi=\xi$. After noticing that $\{a_k\}_{k\in\mathbb{Z}}$ have orthogonal ranges, from the above computations, one has $\sum_{k\in\mathbb{Z}}a^\dagger_ka_k$ converges pointwise strongly to $I - P_\Om$.

If $\ga_m$ had an identity, it would be a selfadjoint projection belonging to the center. But this is impossible by the first half.
\end{proof}
\noindent
Thus, the relationship between $\gar_m$ and $\ga_m$ is given by the next result.
\begin{cor}
\label{noid2}
$\gar_m=\ga_m + \mathbb{C}I$.
\end{cor}
\noindent
Once having established the fine structure of the Monotone $C^*$-algebra $\gar_m$, we pass to the investigation of the convex set consisting of its shift invariant states.
The first step is to show that the fixed point subalgebra is trivial.
\begin{prop}
\label{fixalg}
The the fixed point subalgebra $\gar_m^\mathbb{Z}$ w.r.t. the action of the shift $\a$ is trivial:
$\gar_m^\mathbb{Z}=\mathbb{C}I$.
\end{prop}
\begin{proof}
As $\a(\ga_m)=\ga_m$, by Corollary \ref{noid2} it is enough to see that $\ga_m^\mathbb{Z}=0$. First we note that for each $X\in\ga_m^0$, there exists a sufficiently large
$k(X)\in\mathbb{N}$ such that $\alpha^{k(X)}(X)X=0$.
Now let $0\neq Y\in\ga_m$ such that $\alpha(Y)=Y$. Without loss of generality, we can take $Y$ selfadjoint with $\|Y\|\leq 1$. Let $(Y_n)_n$ be a sequence of polynomials in $\ga^0_m$ such that
$\|Y-Y_n\|<1/n$. Moreover, for each $n\in\bn$ there is $k_n\in\mathbb{N}$ s.t. $\alpha^{k_n}(Y_n)Y_n=0$. As a consequence, since $\|Y_n\|< 1+1/n$ and $\alpha$ is an automorphism,
\begin{align*}
\|Y^2\|=&\|\alpha^{k_n}(Y)Y\|\leq\|(\alpha^{k_n}(Y-Y_n))(Y-Y_n)\|+\|\alpha^{k_n}(Y-Y_n)Y_n\|\\
+&\|\alpha^{k_n}(Y_n)(Y-Y_n)\|
+\|\alpha^{k_n}(Y_n)(Y_n)\|< \frac{1}{n}\big(2+\frac{3}{n}\big)\,,
\end{align*}
which contradicts $Y\neq 0$.
\end{proof}
\noindent
We are now ready to describe all the shift invariant states on the Monotone algebra $\gar_m$. As we have $\gar_m=\ga_m+\bc I$, and the former does not contain any own identity, $\cs(\gar_m)$ is the one-point compactification of all the positive functionals on $\ga_m$ with norm less than $1$. The {\it state at infinity}
$$
\om_\infty(X +cI):=c\,,\quad X\in \ga_m\,,c\in \mathbb{C}\,,
$$
provides such a "point at infinity". It is shift invariant. Together with the Monotone vacuum,
$$
\om(Y):=\langle Y\Om,\Om\rangle\,,\quad Y\in \gar_m\,,
$$
which we recall is also shift invariant,
the structure of such invariant states is described in the following
\begin{thm}
\label{invmon}
The $*$-weakly compact set of shift invariant states on $\gar_m$ is given by
$$
\cs_\bz(\gar_m)=\{(1-\gamma)\om_\infty + \gamma\om\mid \gamma\in [0,1]\}\,.
$$
\end{thm}
\begin{proof}
As any state $\f\in\cs(\gar_m)$ is uniquely determined by its values on the dense subalgebra $\gar_m^0$, we consider a generic element
$$
X=cI+Y+\sum_{i=1}^l\b_ia_{j_i}a_{j_i}^\dagger\,.
$$
Here, $Y\in\ga_m^0$ is a finite combination of reduced words in $\l$-form with lengths $>\!\!0$, and in $\pi$-form with lengths $>\!\!2$.
For the vacuum, we get $\om(X)=c+\sum_{i=1}^l\b_i$,
whereas for the state at infinity, $\om_\infty(X)=c$. Consider now an arbitrary $\f\in\cs_\bz(\gar_m)$. By Proposition \ref{sum2},
$$
\lim_n\bigg(\frac1{n}\sum_{k=0}^{n-1}\a^k(Y)\bigg)=0
$$
in norm. By invariance, we get
\begin{align*}
&\f(Y)=\frac1{n}\sum_{k=0}^{n-1}\f(\a^k(Y))
=\f\bigg(\frac1{n}\sum_{k=0}^{n-1}\a^k(Y)\bigg)\\
=&\lim_n\bigg(\frac1{n}\sum_{k=0}^{n-1}\f\big(\a^k(Y)\big)\bigg)
=\f\left(\lim_n\bigg(\frac1{n}\sum_{k=0}^{n-1}\a^k(Y)\bigg)\right)=0\,.
\end{align*}
Denote $\g:=\f(a_ia^\dagger_i)$, $i\in\bz$ the common value of $\f$ on the reduced words of the form $aa^\dagger$. We get
$$
\f(X)=c+\g\sum_{i=1}^l\b_i=(1-\g)\om_\infty(X)+\g\om(X)\,.
$$
\end{proof}
\noindent
The next Proposition shows that $\gg_m$ and $\gar_m$ are equal. The result is analogous to the Boolean case (cf. \cite{CrFid}, Proposition 7.1), whereas differs from the $q$-deformed cases including the Free one corresponding to $q=0$. In fact, for the latter the von Neumann algebra generated by the annihilators and that generated by their selfadjoint parts act irreducibly and in standard form on the Fock space respectively, see e.g. \cite{BS, DykFid} and the references cited therein.
\begin{prop}
\label{invmon2}
$\gg_m$ and $\gar_m$ coincide.
\end{prop}
\begin{proof}
We show that the generators of $\gar_m$ belong to $\gg_m$.
To this aim we firstly notice that, for each $i\in\mathbb{Z}$,
\begin{equation}
\label{i+1}
a^\dag_{i+1}a_{i+1}=a_ia^\dag_i-a_{i+1}a^\dag_{i+1}\,.
\end{equation}
Indeed, both $a^\dag_{i+1}a_{i+1}$ and $a_ia^\dag_i-a_{i+1}a^\dag_{i+1}$ vanish on $\Om=e_\emptyset$. Furthermore, for $n\in\mathbb{N}$ and $\xi:=
e_{(j_1,j_2,\cdots, j_n)}$, 
$$
a^\dag_{i+1}a_{i+1}\xi=\left\{
        \begin{array}{ll}
          \xi & \text{if}\,\, i+1=j_1\,,\\
          0 & \text{otherwise}\,,
        \end{array}
        \right.
$$
and
$$
(a_ia^\dag_i-a_{i+1}a^\dag_{i+1})\xi=\left\{
        \begin{array}{lll}
          0 & \text{if}\,\, i\geq j_1\,,\\
          \xi & \text{if}\,\, i+1=j_1\,,\\
          0 & \text{if}\,\, i+1< j_1\,.
        \end{array}
        \right.
$$
This implies that $a^\dagger_i\in\gg_m$, $i\in\mathbb{Z}$, since, from \eqref{i+1} and \eqref{b-01b} it follows
$$
s_is_{i+1}^2=(a_i+a^\dag_i)(a_{i+1}a^\dag_{i+1}+ a^\dag_{i+1}a_{i+1})=(a_i+a^\dag_i)a_ia^\dag_i=a^\dag_i\,.
$$
\end{proof}
\noindent
We end the present section by noticing that the condition \eqref{commcz} is not satisfied for the Monotone quantisation, if $U=U_g$, and for the canonical basis $\{e_j\mid j\in\bz\}$ of $\ell^2(\bz)$,
$$
U_ge_j=e_{g(j)}\,,\quad g\in\bp_{\bz}\,, j\in\bz
$$
is the usual representation of $\bp_{\bz}$ on $\ell^2(\bz)$. Thus, the permutations do not act on $\gar_m$ in a canonical way as Bogoliubov automorphisms.

\section{examples}
\label{exa}
\subsection{Anti-Monotone case}
\label{antimon}
The Anti-Monotone quantisation is associated with the Yang-Baxter-Hecke operator $T:\ell^2(\bz)\otimes\ell^2(\bz)\to\ell^2(\bz)\otimes\ell^2(\bz)$ given by
\begin{equation*}
T_{am}(e_i\otimes e_j):=\left\{
\begin{array}{ll}
                       0 & \text{if}\,\, i>j\,,\\
                       -(e_i\otimes e_j) & \text{if}\,\, i\leq j\,.
                     \end{array}
                     \right.
\end{equation*}
It is easy to show that, if one defines the unitary reflection $R$ on the canonical basis of $\ell^2(\bz)$ given by
$$
Re_j:=e_{-j}\,,\quad j\in\bz\,,
$$
we have
$$
T_{am}=(R\otimes R)T_m(R\otimes R)\,.
$$
Denote $a_m$ and $a_{am}$ the Monotone and Anti-Monotone annihilator, respectively. Thus, $\underbrace{R\otimes \cdots \otimes R}_{n\,\, \text{times}}$ acting on 
$\ell^2(\bz)^{\otimes n}$, uniquely defines
a unitary operator
$$
\cf_{T_m,T_{am}}(R):\cf_{T_m}(\ell^2(\bz))\to\cf_{T_{am}}(\ell^2(\bz))
$$
(whose adjoint is $\cf_{T_{am},T_{m}}(R)$)
satisfying
$$
\cf_{T_m,T_{am}}(R)a_m(f)\cf_{T_m,T_{am}}(R)^*=a_{am}(Rf)\,,\quad f\in\ell^2(\bz)\,.
$$
We immediately conclude that all the results of Section \ref{monsec} apply also to the Anti-Monotone case.

\subsection{uniquely mixing dynamical systems}
\label{ecfi}

Due to Lemma \ref{sum1}, we see that it is always possible to associate to any Yang-Baxter-Hecke operator, a concrete $C^*$-dynamical system enjoying the very strong ergodic property of the unique mixing w.r.t. the fixed point subalgebra, even if \eqref{unifbound} is not satisfied. This case covers the most known one, that is the Bose quantisation case for which \eqref{unifbound} does not hold. For such a purpose, denote $E^n_T$ the selfadjoint projection onto the $n$-particle subspace $\ch^{n}_T$, and consider the concrete
$C^*$-algebra $\ge_T$ acting on $\cf_T(\ch)$ and generated by the identity $I$ together with elements of the form
\begin{equation*}
X=E^m_Ta^\sharp(e_{i_1})a^\sharp(e_{i_2})\cdots a^\sharp(e_{i_r})E^n_T\,,
\end{equation*}
where $m,n,r=0,1,2,\dots\,\,.$

Let $ \ch=\ell^2(\bz)$. Suppose that the Yang-Baxter-Hecke operator $T$ satisfies \eqref{commcz}, and in addition that the sum appearing in \eqref{commrul} is a finite one. In this situation, the shift $\a$ acts also on $\ge_T$ by restriction, because all the projections $E^n_T$ are invariant by construction.
\begin{prop}
\label{mixing2}
Let $T$ be a Yang-Baxter-Hecke selfadjoint operator on $\ch=\ell^2(\mathbb{Z})$. Suppose that the sum in \eqref{commrul} is finite, and furthermore \eqref{commcz} holds. Then the $C^*$-dynamical system $(\ge_T, \alpha)$ is uniquely mixing w.r.t. the fixed-point subalgebra.
\end{prop}
\begin{proof}
The proof, which we are going to sketch, follows the same lines of Theorem \ref{mixing}.
By a standard approximation argument, we can reduce the matter to managing only the algebraic generators. Thus,
as the sum in \eqref{commrul} is finite, each nonempty word (not necessarily reduced) in annihilators and creators can be expressed in the Wick form. In addition, by passing to the adjoint, we can suppose that all the involved nontrivial words have the form $E^m_Ta^\dagger(e_{i})Y$, where $Y$ is another word, either starting with some $a^\sharp$ or with some projection
$E^m_T$. Considering a general linear combination $X$ of words as above, and
reasoning as in the proof of Proposition \ref{sum2}, there exists a sufficiently large natural number $n(X)$ such that
$$
\left \|\sum_{h=1}^n \a^{k_h}(X)\right \|\leq \sqrt{n}\|R^{(n(X))}\|\,.
$$
The proof now follows as in Theorem \ref{mixing}.
\end{proof}

\subsection{Bose case}

The case of the CCR corresponding to the Bose case, is covered by considering the flip map $\sigma: \ch\otimes \ch \rightarrow \ch\otimes \ch$ s.t. $\sigma(x \otimes y):=y \otimes x$. The flip map generates in a canonical way a representation of the permutations on the whole $\ch^{\otimes n}$, denoted by an abuse of notation by $\pi$.
In this case,
$T_\text{Bose}=\s$, and the corresponding $P$ are given by (a multiple of) the classical symmetrizator on each $\ch^{\otimes n}$, i.e.
$$
P^{(n)}_\text{Bose}= \sum_{\pi\in\bp_n}\pi\,.
$$
It is well known that (even for the simple case of one degree of freedom for which $\ch=\bc$) the CCR cannot be realised by bounded operators. Thus, the above analysis cannot apply.
However, by considering the Weyl construction for the CCR via the Weyl algebra, it is still possible to investigate the ergodic properties enjoyed by the Bose particles. For the case
$\ch=\ell^2(\bz)$, it is easy to see that the shift and the permutations naturally act on the Weyl CCR algebra. It is well known that there exists plenty of shift invariant states. In addition, by using the results on the Weyl algebra in \cite{LRT} (see also \cite{BR1}), the structure of the symmetric states are covered by the quantum De Finetti Theorem in \cite{St2}.

It is worth noticing that, unlike the dynamical systems based on the Weyl algebra, the particular Bose-like system $(\ge_\text{Bose},\a)$ of the Subsection \ref{ecfi} is uniquely mixing w.r.t. the fixed point subalgebra
$$
\ge_\text{Bose}^\mathbb{Z}\sim C_0(\bn)+\bc\idd\sim C(\bn_\infty)\,,
$$
even if it is unknown to the authors potential physical applications of such a system.

 \subsection{Fermi case}

The Fermi case associated to the so called CAR algebra, also falls into the class we are dealing with. In fact $T_\text{Fermi}=-\sigma$, the latter being the flip unitary, and
$$
P^{(n)}_\text{Fermi}= \sum_{\pi\in\bp_n}\epsilon(\pi)\pi\,,
$$
where $\epsilon(\pi)$ denotes the sign of the permutation $\pi$.
In this case, one recognises that, up to normalisation factors, the $T_\text{Fermi}$-deformed Fock space as the Fermi Fock space over $\ell^2(\mathbb{Z})$, and the unital $C^*$-algebra generated by $\{a_i \mid i\in\mathbb{Z}\}$ as the CAR algebra over $\mathbb{Z}$. It is well known that in such a case, creators and annihilators are bounded, that is for each $i$, $\|a_i\|=\|a^\dagger_i\|=1$ (see, e.g. \cite{BR1}, Proposition 5.2.2), even if \eqref{unifbound} does not hold. Indeed, suppose that \eqref{unifbound} holds. Then it would exist $M_T<\infty$ s.t. for each $n$, $\|R^{(n)}\|\leq M_T$. By \eqref{pnorm} and \eqref{Pn} one should obtain that for any $n\in\mathbb{N}$,
$$
(n+1)!=\big\|P_\text{Fermi}^{(n+1)}\big\|=\big\|(I\otimes P_\text{Fermi}^{(n)})R^{(n+1)}\big\|\leq n!M_T\,,
$$
in contradiction to the unboundedness of natural integers.

For $\ch=\ell^2(\bz)$, the shift as well as the permutation group act on $\gar_\text{Fermi}$. As for the Bose case, it is well known that there exists plenty of shift invariant states. The structure of the symmetric states is instead fully covered by the quantum De Finetti Theorem in \cite{CrF}.

\section{E-mixing for the Boolean shift}
\label{boolsec}

We investigate in some detail the ergodic properties of the Boolean $C^*$-algebra, even if the general results of Section \ref{sec3} cannot be directly applied to this case.

Let $\ch$ be a complex Hilbert space. Recall that the Boolean Fock space over $\ch$ is given by $\G(\ch):=\mathbb{C}\oplus \ch$,
where the vacuum vector $\Om$ is $(1,0)$.
On $\Gamma(\ch)$ we define the creation and annihilation operators (denoted again with $a^\dagger, a$), respectively given for $f\in \ch$, by
$$
a^\dagger(f)(\alpha\oplus g):=0\oplus \alpha f,\,\,\,\, a(f)(\alpha\oplus g):=\langle g,f\rangle_\ch \oplus 0,\,\,\, \alpha\in\mathbb{C},\, g\in\ch.
$$
They are mutually adjoint, and satisfy the following relations
\begin{align}
\label{boolrel}
&a(f)a(g)=a^\dagger(f)a^\dagger(g)=0\,,\quad f,g\in\ch\, \nn\\
&a(f)a^\dagger(g)=\langle g,f\rangle P_\Om=\langle g,f\rangle (I - \sum_{j\in J}a^\dagger(e_j)a(e_j))\,,
\end{align}
for any orthonormal basis $\{e_j\mid j\in J\}$ of the involved Hilbert space. The Boolean Fock space is also obtained as a $T$-deformed one, simply by taking $T=-I$, that is $\Gamma(\ch)=\cf_{-I}(\ch)$. Therefore, $P_\text{Boole}^{(n)}=0$ whenever $n\geq 2$. The commutation \eqref{boolrel} is just the analogous of \eqref{commrul} for
$T=-I$. If $\text{dim}(\ch)=+\infty$, the sum in \eqref{boolrel} is infinite and the result in Section \ref{sec3} cannot be directly applied for the investigation of the ergodic properties of the Boolean $C^*$-algebra $\gar_\text{Boole}$.

We put $\ch=\ell^2(\mathbb{Z})$, and in this situation $\Gamma(l^2(\mathbb{Z}))=\ell^2(\{\#\}\cup\mathbb{Z})$, where the vacuum vector and the vacuum state are $\Om=e_\#$
and $\om_\#:=\langle\,{\bf \cdot}\,e_\#,e_\#\rangle$, respectively. In \cite{CrFid} it was shown that
$$
\ck(\G(\ell^2(\{\#\}\cup\bz))+ \mathbb{C}I=\gar_\text{Boole}=:\gpb\,,
$$
and $\ck(\G(\ell^2(\{\#\}\cup\bz))$ denotes the compact linear operators on $\ell^2(\{\#\}\cup\bz)$.
In addition, the $C^*$-subalgebra generated by the selfadjoint part of the annihilators coincides with the Boolean one: $\gg_\text{Boole}=\gar_\text{Boole}$. The shift, as well as the permutations $\bp_\bz$ naturally act on $\gpb$ as Bogoliubov automorphisms. Denote $\gpb^{\bp_\bz}$, $\gpb^{\bz}$ the fixed point subalgebras w.r.t. the actions of the permutations and the shift, that is the exchangeable and the invariant $C^*$-subalgebra, respectively. We get
\begin{prop}
\label{fixal}
For the fixed point subalgebras $\gpb^{\bp_\bz}$, $\gpb^{\bz}$, one has
$$
\gpb^{\bp_\bz}=\gpb^{\bz}= \mathbb{C}P_\# \oplus \mathbb{C}P_\#^\bot\,,
$$
where, for each $i$, $P_\#=a_ia^\dagger_i$ denotes the orthogonal projection onto the linear span of $e_\#$.
\end{prop}
\begin{proof}
Denote by $\{V_g\mid g\in\bp_\bz\}$ and $V$ the unitary implementations of the permutations $\bp_\bz$ and the shift on $\ell^2(\bz)$, respectively. The corresponding implementation of the permutations and the shift on the Boolean Fock space are given, respectively, by
$$
U_g=P_\#\oplus V_g\,,\,\,\,g\in\bp_\bz\,,\quad U=P_\#\oplus V\,,
$$
see \cite{F} for more details.
It is easy to show that,  first (cf. Proposition 3 of
\cite{Fbo})
$$
\cb(\ell^2(\bz))^{\bp_\bz}=\{V_\t\mid\t\,\,\text{transposition}\}'=\bc I_{\ell^2(\bz)}\,,
$$
and second (cf. Lemma 7.2 of \cite{CrFid}),
$$
(\ck(\ell^2(\{\#\}\cup\bz)))^{\bz}=\bc P_\#\,,
$$
which lead to the assertion.
\end{proof}
\noindent
Denote $E$ the conditional expectation onto $\gpb^{\bz}$ given by
$$
E(A+bI):=\langle Ae_\#,e_\#\rangle P_\#+bI\,,\quad A\in\ck(\ell^2(\{\#\}\cup\bz))\,,\,\,b\in\bc\,.
$$
Notice that it is invariant both for the action of the shift and the permutations.
Although the fixed point subalgebra is non trivial, the following result ensures that, in the Boolean case, the $E^\bz$-mixing property for the shift holds.
\begin{prop}
\label{bshi}
The $C^*$-dynamical system $(\gpb, \a)$ is $E^\bz$-mixing, with $E=E^\bz$ the unique invariant conditional expectation onto the fixed point subalgebra.
\end{prop}
\begin{proof}
Let $\{e_j\mid j\in\{\#\}\cup\bz\}$ be the canonical basis of $\ell^2(\{\#\}\cup\bz)$. By a standard approximation argument and Proposition 2.3 of \cite{F22}, it is enough to show that, for each subsequence $\{l_1,l_2,\dots,l_k,\dots\}\subset\bn$,
$$
\lim_n\frac1{n}\sum_{k=1}^n\a^{l_k}(A)=\langle Ae_\#,e_\#\rangle P_\#\,,
$$
where $A$ is the rank-one operator of the form $A=\langle\,{\bf\cdot}\,,e_i\rangle e_j$, $i,j\in\{\#\}\cup\bz\}$, and the limit is meant in norm. If $A=P_\#$, it is invariant and the above condition is automatically satisfied. For $A=\langle\,{\bf\cdot}\,,e_\#\rangle e_j$, $j\in\bz$, and a unit vector $\xi\in\ell^2(\{\#\}\cup\bz)$, we get by orthogonality,
$$
\bigg\|\sum_{k=1}^n\a^{l_k}(A)\xi\bigg\|=\bigg\|\sum_{k=1}^n\langle\xi,e_\#\rangle e_{j+l_k}\bigg\|
=\sqrt{\sum_{k=1}^n|\langle\xi,e_\#\rangle|^2}\leq\sqrt{n}\,.
$$
The same is true for $A=\langle\,{\bf\cdot}\,,e_i\rangle e_\#$, $i\in\bz$, by taking the adjoint. For $A=\langle\,{\bf\cdot}\,,e_i\rangle e_j$, $i,j\in\bz$, we again get
$$
\bigg\|\sum_{k=1}^n\a^{l_k}(A)\xi\bigg\|=\bigg\|\sum_{k=1}^n\langle\xi,e_{i+l_k}\rangle e_{j+l_k}\bigg\|
=\sqrt{\sum_{k=1}^n|\langle\xi,e_{i+l_k}\rangle|^2}\leq\sqrt{n}\,.
$$
The proof follows by dividing by $n$ and taking the limit.
\end{proof}
\begin{cor}
For the set of the stationary states, we get
$$
\cs_\bz(\gpb)=\{\f\circ E\mid\f\in\cs(\gpb^\bz)\}\,.
$$
\end{cor}
\vskip.3cm
\begin{rem}
Concerning the action $\b$ of the permutation group, following the same lines of Proposition \ref{bshi} we can show that
$$
\lim_{J\uparrow\bz}\frac1{|J|!}\sum_{g\in\bp_J}\b_g(A)=E(A)\,,\quad A\in\gpb\,,
$$
where $\{J\mid J\subset\bz\}$ is the direct net made of all the finite subset of $\bz$, and the limit is meant in norm. In addition, for the symmetric states we again have
$$
\cs_{\bp_\bz}(\gpb)=\{\f\circ E\mid\f\in\cs(\gpb^{\bp_\bz})\}\,.
$$
\end{rem}
\vskip.3cm
\noindent
Relatively to the symmetric and stationary states on $\gpb$, collecting together the above results and Proposition 7.3 of \cite{CrFid}, one obtains that all of them are exactly those lying on a segment. Namely,
$$
\cs_\bz(\gpb)=\cs_{\bp_\bz}(\gpb)=\{\g\om_\#+(1-\g)\om_\infty\mid\g\in[0,1]\}\,,
$$
where
$$
\om_\infty(A+bI):=b\,,\quad A\in \ck(\ell^2(\{\#\}\cup\bz))\,, \,\,b\in\bc\,.
$$

\end{document}